\documentclass[12pt]{article}

\usepackage{dsfont}
\usepackage{amsthm,amsmath,amssymb}
\usepackage{enumerate}
\usepackage{enumerate}
\usepackage{indentfirst}
\newcommand{\qbc}[2]{\left[\genfrac{}{}{0pt}{}{#1}{#2}\right]}

\theoremstyle{plain}
\newtheorem{theorem}{Theorem}[section]
\newtheorem{lemma}[theorem]{Lemma}

\newtheorem{example}[theorem]{Example}

\title{A note on the distribution of major index for Schr\"{o}der paths}
\date{}

\author{Xiaomei Chen}
\begin{document}
\maketitle

%\address{School of Mathematics and Computational Science, Hunan University of Science and
%Technology, Xiangtan 411201, China}

%\vspace{-10mm}
\begin{abstract}
Bonin, Shapiro and Simion (1993) gave two formulas on the distribution of major index for Schr\"{o}der paths, and proved their result for the case $E<D<N$. In this short note, we correct an error in their proof, and give a complete proof for all cases.
\end{abstract}

%% main text
\section{Introduction}

For the notation and terminology below on lattice paths, see \cite{Bonin-Shapiro-Simion}. Let $Del(m,n,l)$ denote the set of Delannoy paths from $(0,0)$ to $(m,n)$ with $l$ steps, where a Delannoy path is a lattice path using only the three steps (1,0), (1,1) and (0,1). A Schr\"{o}der $n-$path is a Delannoy path from $(0,0)$ to $(n,n)$ which never goes above the diagonal line $y=x$, and we use $Sch_{L}(n,l)$ to denote the collection of all Schr\"{o}der $n-$paths with $l$ steps.

In the following, we use $E$, $D$ and $N$ to denote the three steps (1,0), (1,1) and (0,1) respectively, and represent a Delannoy path of length $l$ by a word $W=w_{1}w_{2}\cdots w_{l}$ over the alphabet set $\{E, D, N\}$. Given a linear ordering of $\{E, D, N\}$, $i (1\leq i\leq n-1)$ is called a descent of $W$ if $w_{i}>w_{i+1}$. We use $D(W)$ to denote the set of all descents of $W$, and define the major index of $W$ by $\mathrm{maj}(W):=\sum_{i\in D(W)}i$.

Bonin et al. studied the major index for Schr\"{o}der paths and gave the following result.
\begin{theorem}\cite{Bonin-Shapiro-Simion}
\label{thm:11}
For given $n$, $l$ and linear ordering of $\{E, D, N\}$, the distribution of the statistic maj over $Sch_{L}(n,l)$ is
\begin{equation}
\begin{aligned}
MSch_{L}(n,l;q):&=\sum_{W\in Sch_{L}(n,l)}q^{\mathrm{maj}(W)}\\
&=\frac{1}{[l-n+1]}\qbc{2(l-n)}{l-n}\qbc{l}{2n-l},\,if \,E<N,
\end{aligned}
\end{equation}
and
\begin{equation}
\begin{aligned}
MSch_{L}(n,l;q):&=\sum_{W\in Sch_{L}(n,l)}q^{\mathrm{maj}(W)}\\
&=\frac{q^{l-n}}{[l-n+1]}\qbc{2(l-n)}{l-n}\qbc{l}{2n-l},\,if \,E>N.
\end{aligned}
\end{equation}
\end{theorem}

Bonin et al. gave a detailed proof of the above result for the case $E<D<N$, and omitted proof of other cases. In their proof, a path $p\in Del(n,n,l)$ is called a 'bad' one if it runs above the line $y=x$, and a correspondence $\psi$ from the set consisting of all bad paths in $Del(n,n,l)$ to the set $Del(n+1,n-1,l)$ is defined as follows.

For a 'bad' path $W=w_1w_2\cdots w_l\in Del(n,n,l)$, let $w_i$ be the first step of $W$ running above the line $y=x$, and let $w_j$ be the last element of the sequence of consecutive $N$ beginning at $w_i$. Then $\psi(W)$ is defined to be the path obtained from $W$ by replacing $w_j$ with $E$. The correctness of the proof given in \cite{Bonin-Shapiro-Simion} relies on the statement that $\psi$ is a bijection, which however is not true. See the following example for instance.

\begin{example}
Let $w=NENNEE$ and $w'=EENNNE$ be two bad paths in $Del(3,3,6)$. Then $\psi(w)=\psi(w')=EENNEE$.
\end{example}

\section{Proof of Theorem \ref{thm:11}}
Let $BDel(m,n,l)$ denote the set of all 'bad' paths in $Del(m,n,l)$. Given a linear ordering of $\{E,D,N\}$, we define
$$MDel(m,n,l;q):=\sum_{W\in Del(m,n,l)}q^{\mathrm{maj}(W)}$$
and
$$MBDel(m,n,l;q):=\sum_{W\in BDel(m,n,l)}q^{\mathrm{maj}(W)}$$
to be the distributions of the maj statistic over $Del(m,n,l)$ and $BDel(m,n,l)$ respectively.

For a path $W=w_1w_2\cdots w_l \in Del(m,n,l)$, the depth of $w_i(1\leq i\leq l)$ is defined to be the difference between the number of $N$ and the number of $E$ in the subpath $w_1w_2\cdots w_i$. By extending the technique applied to Catalan paths in \cite{F¨¹rlinger-Hofbauer}, we obtain the following result.

\begin{lemma}
\label{lem:21}
For given $n,l$ and linear ordering of $\{E,D,N\}$, we have
\begin{equation}
\begin{aligned}
MBDel(n,n,l;q)=qMDel(n+1,n-1,l;q),\,if \,E<N,
\end{aligned}
\end{equation}
and
\begin{equation}
\begin{aligned}
MBDel(n,n,l;q)=MDel(n+1,n-1,l;q),\,if \,E>N.
\end{aligned}
\end{equation}
\end{lemma}
\begin{proof}
We prove Lemma $\ref{lem:21}$ by constructing a bijection
$$\varphi: BDel(n,n,l)\rightarrow Del(n+1,n-1,l)$$
as follows. Given a path $W=w_1w_2\cdots w_l \in BDel(n,n,l)$, let $w_k$ be its first deepest step. We denote by
$$W_1:=w_{k-r}w_{k-r+1}\cdots w_{k-1}w_kw_{k+1}\cdots w_{k+s}$$
the subpath of $W$ such that $w_{k-r-1}\neq D$, $w_{k+s+1}\neq D$, and $w_{k-i}=w_{k+j}=D$ for $1\leq i\leq r$ and $1\leq j\leq s$. Then $\varphi(W)$ is defined as follows.
\begin{enumerate}[(1)]
  \item If $E<D<N$, $N<E<D$ or $D<N<E$, $\varphi(W)$ is obtained from $W$ by replacing $w_k$ with $E$.
  \item If $E<N<D$, $D<E<N$ or $N<D<E$, $\varphi(W)$ is obtained from $W$ according to the following two cases:
  \begin{itemize}
    \item when $r\geq 1$, replacing $W_1$ with the path $\underbrace{DD\cdots D}_{r-1}E\underbrace{DD\cdots D}_{s+1}$;
    \item when $r=0$, replacing $W_1$ with the path $\underbrace{DD\cdots D}_{s}E$.
  \end{itemize}
\end{enumerate}

By the definition of $w_k$ and $W_1$, we must have $w_k=N$, $w_{k+s+1}=E$, and $w_{k-r-1}=N$ if $k-r-1\geq 1$. Then it is not difficult to verify that
$$\mathrm{maj}(W)=\left\{
           \begin{array}{ll}
             \mathrm{maj}(\varphi(w))+1, & \hbox{if $E<N$;} \\
             \mathrm{maj}(\varphi(w)), & \hbox{if $E>N$.}
           \end{array}
         \right.
$$
For instance, if $E<N<D$ and $r=0$, then $D(\varphi(W))=D(W)-\{k+s\}\cup\{k-s-1\}$, which implies that $\mathrm{maj}(W)=\mathrm{maj}(\varphi(W))+1$.

Thus to complete the proof of Lemma \ref{lem:21}, it is enough to show that $\varphi$ is a bijection. Given a path $W=w_1w_2\cdots w_l\in Del(n+1,n-1,l)$, let $w_{k-1}$ be the last of its deepest step, where the depth of $w_0$ is defined to be 0. Then we must have $w_k=E$. We denote by
$$W_1:=w_{k-r}w_{k-r+1}\cdots w_{k-1}w_kw_{k+1}\cdots w_{k+s}$$
the subpath of $W$ such that $w_{k-r-1}\neq D$, $w_{k+s+1}\neq D$, and $w_{k-i}=w_{k+j}=D$ for $1\leq i\leq r$ and $1\leq j\leq s$. Then the path $\varphi^{-1}(W)$ is constructed as follows.
\begin{enumerate}[(1)]
  \item If $E<D<N$, $N<E<D$ or $D<N<E$, $\varphi^{-1}(W)$ is obtained from $W$ by replacing $w_k$ with $N$.
  \item If $E<N<D$, $D<E<N$ or $N<D<E$, $\varphi^{-1}(W)$ is obtained from $W$ according to the following two cases:
  \begin{itemize}
    \item when $s\geq 1$, replacing $W_1$ with the path $\underbrace{DD\cdots D}_{r+1}N\underbrace{DD\cdots D}_{s-1}$;
    \item when $s=0$, replacing $W_1$ with the path $N\underbrace{DD\cdots D}_{r}$.
  \end{itemize}
\end{enumerate}

It is obvious that the above construction gives the inverse of $\varphi$, thus we completes the proof.
\end{proof}

We are now ready to give the proof of Theorem \ref{thm:11}.
\begin{proof}
By a result of MacMahon(\cite{MacMahon}) on the distribution of maj over all permutations of a multiset, we have
$$MDel(m,n,l;q)=\qbc{l}{l-m,l-n,m+n-l}.$$
It is obvious that
$$MSch_{L}(n,l;q)=MDel(n,n,l;q)-MBDel(n,n,l;q).$$
Therefore by Lemma \ref{lem:21}, for the case when $E<N$, we have
\begin{equation*}
\begin{aligned}
MSch_{L}(n,l;q)&=\qbc{l}{l-n,l-n,2n-l}-q\qbc{l}{l-n-1,l-n+1,2n-l}\\
&=\frac{1}{[l-n+1]}\qbc{2(l-n)}{l-n}\qbc{l}{2n-l},
\end{aligned}
\end{equation*}
 and for the case when $E>N$, we have
\begin{equation*}
\begin{aligned}
MSch_{L}(n,l;q)&=\qbc{l}{l-n,l-n,2n-l}-\qbc{l}{l-n-1,l-n+1,2n-l}\\
&=\frac{q^{l-n}}{[l-n+1]}\qbc{2(l-n)}{l-n}\qbc{l}{2n-l}.
\end{aligned}
\end{equation*}
\end{proof}

%% References
%%
%% Following citation commands can be used in the body text:
%% Usage of \cite is as follows:
%%   \cite{key}          ==>>  [#]
%%   \cite[chap. 2]{key} ==>>  [#, chap. 2]
%%   \citet{key}         ==>>  Author [#]

%% References with bibTeX database:

%\bibliographystyle{model1b-num-names}
%\bibliography{<your-bib-database>}

\begin{thebibliography}{00}
\bibitem{Bonin-Shapiro-Simion}
Bonin J, Shapiro L, Simion R. Some q-analogues of the Schr\"{o}der numbers arising from combinatorial statistics on lattice paths[J]. Journal of Statistical Planning and Inference, 1993, 34(1): 35-55.
\bibitem{F¨¹rlinger-Hofbauer}
F\"{u}rlinger J, Hofbauer J. q-Catalan numbers[J]. Journal of Combinatorial Theory, Series A, 1985, 40(2): 248-264.
\bibitem{MacMahon}
MacMahon P A. Combinatory analysis(Vol. 2)[M]. Cambridge University Press, Cambridge, 1918. Reprinted by Chelsea, New York,
 1960.
\end{thebibliography}

%% Authors are advised to submit their bibtex database files. They are
%% requested to list a bibtex style file in the manuscript if they do
%% not want to use model1b-num-names.bst.

%% References without bibTeX database:

\noindent{\emph{Address}: School of Mathematics and Computational Science, Hunan University of Science and Technology, Xiangtan 411201, China.}\\
\noindent{\emph{E-mail address}: xmchen@hnust.edu.cn}
\end{document}